\documentclass[10pt]{article}

\usepackage{amsfonts}
\usepackage{amsmath}
\usepackage{amssymb}
\usepackage{amsthm}
\usepackage{enumerate}
\usepackage{mdwlist}
\usepackage{mathrsfs}
\usepackage{bussproofs}
\usepackage{graphicx}
\usepackage{latexsym}
\usepackage{titlesec}
\usepackage{stmaryrd}
\usepackage{verbatim}
\usepackage{tcolorbox}
\tcbset{colback=white}

\titleformat*{\section}{\large\bfseries}
\titleformat*{\subsection}{\normalsize\bfseries}
\titleformat*{\subsubsection}{\normalsize\em}

\newtheorem{theorem}{Theorem}[section]

\newtheorem{lemma}[theorem]{Lemma}
\newtheorem{corollary}[theorem]{Corollary}

\theoremstyle{remark}
\newtheorem{definition}[theorem]{Definition}

\newtheorem{remark}[theorem]{Remark}

\usepackage{mathtools}

\let\norm\undefined 
\DeclarePairedDelimiter\norm{\lVert}{\rVert}

\newcommand{\RR}{\ensuremath{\mathbb{R}}}

\newcommand{\NN}{\ensuremath{\mathbb N}}

\newcommand{\ceil}[1]{\lceil {#1} \rceil}
\newcommand{\pair}[1]{\langle {#1} \rangle}

\input xy
\xyoption{all}

\title{Rates of convergence for iterative solutions of equations involving set-valued accretive operators}
\author{Ulrich Kohlenbach and Thomas Powell\\[0.2cm] 
Department of Mathematics \\
Technische Universit\"at Darmstadt\\ 
Schlossgartenstra\ss{}e 7\\ 64289 Darmstadt, Germany \\ 
\footnotesize Email: $\{$kohlenbach,powell$\}$@mathematik.tu-darmstadt.de}

\date{April 22, 2020}

\begin{document}
\maketitle

\begin{abstract}
This paper studies proofs of strong convergence of various iterative algorithms for computing the unique zeros of set-valued accretive operators that also satisfy some weak form of uniform accretivity at zero. More precisely, we extract explicit rates of convergence from these proofs which depend on a \emph{modulus} of uniform accretivity at zero, a concept first introduced by A. Koutsoukou-Argyraki and the first author in 2015. Our highly modular approach, which is inspired by the logic-based proof mining paradigm, also establishes that a number 
of seemingly unrelated convergence proofs in the literature are actually instances of a common pattern. 
\end{abstract}

\section{Introduction}

The problem of approximating zeros of accretive set-valued operators 
$A:X\to 2^X$ has been widely studied since the 70's. This is primarily due to the  
importance of these operators in modelling abstract Cauchy problems such as evolution equations (see e.g. 
\cite{Barbu(76),Barbu(10),Reich(1980)}), as well as - for Hilbert spaces $H=X$ (and under the name {\it monotone} operators) - 
their relevance in convex optimization for the 
computation of minima of lower semi-continuous functions $f$, where then $A=\partial f$ 
is the 
subdifferential of $f$ (see e.g. \cite{BauschkeCombettes}) and the 
zero set of $A$ coincides with the set of minimizers of $f.$

In the context of Hilbert spaces, a standard tool for approximating a zero 
of $A$ is the famous Proximal Point Algorithm (PPA) (due to 
\cite{Mar70,Roc76}) which iterates, for varying coefficients $\lambda_n>0$ satisfying appropriate conditions, 
the (single-valued and 
firmly nonexpansive) resolvent 
$J_{\lambda A}=(I+\lambda A)^{-1}$ of $A:$ 
\[ x_{n+1} =J_{\lambda_n A}x_n.\]
The PPA is studied in the context of uniformly convex Banach spaces in 
\cite{BruckReich} but has only recently been 
investigated quantitatively in this setting (\cite{Kohlenbach(PPA)} and 
- for $\lambda_n:=\lambda>0$ - \cite{Koutsoukou}).

For arbitrary Banach spaces, various different types of iterations 
from metric fixed point theory have been used to compute zeros, such as the Krasnoselski-Mann or Ishikawa-type 
iterations. Just as for the PPA, in general these iterations converge only weakly (see e.g. \cite{Bauschkeetal}), and even when strong 
convergence holds (e.g. in the finite dimensional case, for `Halpern-type' 
or `Bruck-type'
modifications, or the operator being odd, see e.g. 
\cite{Reich(1978),Bauschkeetal}), 
there is usually - already for $X=\RR$ - 
no computable (in the sense of Church-Turing) rate of convergence 
(see e.g. \cite{Neumann}).

This situation changes when $A$ satisfies some form of {\it strong} accretivity, which ensures that $A$ has at most one zero $0\in Aq$. General theorems from 
logic guarantee, for a broad range of situations, 
that in the presence of uniqueness one can use quantitative data from the 
uniqueness proofs (e.g. so-called moduli of uniform uniqueness) to give 
rates of convergence for procedures which compute approximate solutions to 
problems (such as finding zeros or fixed points). For all this see e.g. 
\cite{Kohlenbach(book)} (a generalization of the concept of `modulus of 
uniqueness' to the non-unique case, a so-called modulus of regularity, 
also gives a rate of convergence of Fej\'er monotone sequences 
which has been used in different forms many times in the 
literature, see \cite{KohlenbachLopezNicolae} and note that e.g. 
the `uniform convergence condition' 
on $A$ formulated in \cite{NevanlinnaReich} states the existence of a 
special Lipschitz-H\"older type form of a modulus of 
regularity for $zer\,A$). 

Most forms of strong (quasi-)accretivity are stronger and more 
restricted instances of what 
is called {\it uniform accretivity at zero} in \cite[Definition 10]{KohKou(2015.0)}, which is given a quantitative form via a {\it modulus $\Theta$ 
of accretivity at zero}. 

The purpose of this paper is two-fold:
\begin{enumerate} 
\item
to show that in typical cases of known 
strongly convergent algorithms computing the 
unique zero of a strongly accretive operator $A,$ one can extract from the 
convergence proof an explicit rate of convergence in terms of a modulus $\Theta$ of accretivity at zero;
\item 
to provide, using the concept of uniform accretivity at zero together with the logical
analysis of the convergence proofs, a modular and unified account of strong convergence results in the literature which at first glance appear unrelated.
\end{enumerate}
This is exemplified by selecting as test cases the implicit iteration schema 
from \cite{AlbReiSho(2002.0)} together with the explicit Ishikawa-type schemes 
used in \cite{MooNno(2001.0)} and in \cite{Chang} (the latter paper 
being further generalized e.g. in \cite{Lin(2004.0)}). In particular, we recover as special cases the quantitative results 
in \cite{AlbReiSho(2002.0)}. \\[1mm] 
In the case of the Ishikawa-type schemes the conditions on the scalars are so liberal that the Krasnoselskii–Mann
iteration scheme is included as a special case. This is possible since our pseudocontractive operators 
$I-A$ arrive from {\it uniformly} accretive operators $A$ (see 
Lemma \ref{lem-pseudo} and the comment before the lemma). For 
general pseudocontractions the Krasnoselski-Mann schema is known 
to fail to converge already for Lipschitzian pseudocontractive 
selfmappings of compact subsets of a Hilbert space 
while the Ishikawa schema does converge strongly in this situation 
under suitable conditions on the scalars (see \cite{Ishikawa(74)}). \\[1mm] 
 Whereas the main convergence theorems in 
\cite[Theorems 2.1 and 4.1]{AlbReiSho(2002.0)} 
hold in arbitrary Banach spaces and without any continuity assumption on $A,$  
the convergence results in \cite[Theorem 4.1]{Chang} and 
\cite[Theorem 2.2]{MooNno(2001.0)} use the uniform continuity 
of $A$ (w.r.t. the Hausdorff metric) while \cite[Theorem 4.2]{Chang} and 
\cite[Theorem 2.1]{Lin(2004.0)} instead  
use that $X$ is uniformly smooth. Although the assumptions on $A$ being 
uniformly continuous and, respectively, on $X$ being 
uniformly smooth are very different, it turns out they can both be seen 
as instances of the same 
technical lemma. The 
rates of convergence we extract in these cases then also depend 
(in addition to 
$\Theta$) on moduli of uniform continuity for $A$ and, respectively, for the duality 
mapping of $X$, where in the latter case such a modulus can 
be computed in terms of a modulus of uniform smoothness for $X$ (see 
\cite{KohLeu(2012.1)}).

The various forms of strong 
(quasi-)accretivity used in the aforementioned results are all covered 
by mostly more 
restrictive versions 
of our concept of uniform accretivity at zero 
(note that \cite{AlbReiSho(2002.0)}
uses {\it uniform accretivity} to denote a concept which is much 
more restrictive than our notion of uniform accretivity at zero even when 
we drop the restriction `at zero' as it corresponds to 
{\it $\psi$-strong accretivity} as defined in Definition \ref{def-unacc}.(a)
with $\psi$ additionally assumed to be strictly increasing).
Therefore our results strengthen various convergence theorems 
not just quantitatively but also qualitatively.

Since the convergence proofs we study all apply to situations where $A$ can be shown to have a unique zero, in our quantitative results we always \emph{assume} both the existence of a zero and well-definedness of the approximating sequence at hand, which typically allows us to omit certain extra assumptions made in the original papers.

Although no concepts or methods from logic are mentioned explicitly 
in this paper, our approach has been motivated by the tools of the 
\emph{proof mining} program which uses logic-based proof transformations for 
the extraction of effective bounds from prima facie noneffective proofs 
(see \cite{Kohlenbach(book)}). In the case of the proximal point 
algorithm, this approach - again based on the concept of uniform accretivity 
(specialized to the monotone case in Hilbert spaces) - 
has been used in \cite{LeuNiSip} and in the context of uniformly convex 
Banach spaces in \cite{Kohlenbach(PPA)}. For a recent survey on proof 
mining in general see \cite{Kohlenbach(ICM)}.

\section{Preliminaries}
\label{sec-prelim}
$\NN:=\{ 0,1,2,3,\ldots\}$ denotes the set of nonnegative integers. \\ 
Throughout this paper, $X$ will be a real Banach space with dual space $X^\ast$. The normalized duality mapping $J:X\to 2^{X^\ast}$ is defined by
\begin{equation*}
J(x):=\{j\in X^\ast\; : \; \pair{x,j}=\norm{x}^2=\norm{j}^2\}.
\end{equation*}
We will make frequent use of the following well-known geometric inequality.
\begin{lemma}
\label{lem-subdif}
For all $x,y\in X$ and $j\in J(x+y)$ we have
\begin{equation*}
\norm{x+y}^2\leq \norm{x}^2+2\pair{y,j}.
\end{equation*}
\end{lemma}
\begin{proof}
Let $j\in J(x+y)$. Then
\begin{equation*}
\norm{x+y}^2=\pair{x+y,j}\leq \norm{x}\cdot\norm{x+y}+\pair{y,j}\leq \tfrac{1}{2}(\norm{x}^2+\norm{x+y}^2)+\pair{y,j}
\end{equation*}
and the result follows.
\end{proof}
A mapping $A:X\to 2^X$ will be called an operator on $X$. The domain of $A$ is defined by $D(A):=\{x\in X\; : \; Ax\neq\emptyset\}$. We sometimes write $(x,u)\in A$ for $u\in Ax$. The range $R(A)$ of $A$ is defined as $R(A):=\{ y\in X\,:\,
\exists x\in X(y\in Ax)\}.$
%
\subsection{Accretive operators}

For a detailed survey of the various notions of accretivity, including quantitative forms which come equipped with moduli, the reader is encouraged to consult \cite[Section 2.1]{KohKou(2015.0)}. Here, we simply outline the key definitions which play a role in the present paper.
\begin{definition}
\label{def-acc}
An operator $A$ is said to be \emph{accretive} if for all $u\in Ax$ and $v\in Ay$ there exists some $j\in J(x-y)$ such that $\pair{u-v,j}\geq 0$.
\end{definition}
The notion of accretivity was independently introduced (in a slightly different but equivalent form) by Browder \cite{Browder}, Kato \cite{Kato} and Komura \cite{Komura}. However, convergence proofs of the kind we study here typically appeal to various stronger, uniform forms of accretivity:
\begin{definition}
\label{def-unacc}
\begin{enumerate}[(a)]

\item Let $\psi:[0,\infty)\to [0,\infty)$ be a continuous function with $\psi(0)=0$ and $\psi(x)>0$ for $x>0$. Then an operator $A:D(A)\to 2^X$ is said to be \emph{$\psi$-strongly accretive} if
\begin{equation*}
\forall (x,u),(y,v)\in A\; \exists j\in J(x-y)\; (\pair{u-v,j}\geq \psi(\norm{x-y})\norm{x-y}).
\end{equation*}

\item Let $\phi:[0,\infty)\to [0,\infty)$ be a continuous function with $\phi(0)=0$ and $\phi(x)>0$ for $x>0$. Then an operator $A:D(A)\to 2^X$ is said to be \emph{uniformly $\phi$-accretive} if
\begin{equation*}
\forall (x,u),(y,v)\in A\; \exists j\in J(x-y)\; (\pair{u-v,j}\geq \phi(\norm{x-y})).
\end{equation*}

\end{enumerate}
\end{definition}
In the case of $\psi$-strongly accretive operators, $\psi$ is often assumed 
to be strictly increasing in addition (see e.g. \cite{AlbReiSho(2002.0)}). 

It turns out that for all of the results we study in this paper, the above notions can be replaced by the following more general property of being \emph{uniformly accretive at zero}, introduced by Garc\'{i}a-Falset in \cite{GarciaFalcet} and given a quantitative form by the first author in \cite{KohKou(2015.0)}.
\begin{definition}
An accretive operator $A:D(A)\to 2^X$ with $0\in Aq$ is said to be uniformly accretive at zero if
\begin{equation*}
(\ast) \ \ \ \begin{aligned}
&\forall \varepsilon,K>0\; \exists \delta>0\; \forall (x,u)\in A\\
&(\norm{x-q}\in [\varepsilon,K]\to \exists j\in J(x-q)\;(\pair{u,j}\geq \delta)).
\end{aligned}
\end{equation*}
Moreover, any function $\Theta_{(\cdot)}(\cdot):(0,\infty)\times (0,\infty)\to (0,\infty)$ such that $\delta:=\Theta_K(\varepsilon)$ satisfies $(\ast)$ for all $\varepsilon,K>0$ is called a \emph{modulus of uniform accretivity at zero} for $A$.
\end{definition}
In particular, we observe that if $A$ is uniformly $\phi$-accretive, a modulus of uniform accretivity at zero for $A$ is given by
\begin{equation*}
\Theta_K(\varepsilon):=\inf\{\phi(x)\; : \; x\in [\varepsilon,\max\{ 
\varepsilon,K\}]\}.
\end{equation*}
In the case where $\phi$ is also strictly increasing, we can simply let $\Phi_K(\varepsilon):=\phi(\varepsilon)$.
\begin{remark}
Though technically speaking, moduli of uniform accretivity at zero are defined relative to some given $q\in D(A)$ with $0\in Aq$, one can actually show that such a $q$, if it exists, is necessarily unique. Moreover, a modulus of uniqueness for $q$ can be constructed in terms of a modulus of uniform accretivity at zero, as is made precise in \cite[Remark 2]{KohKou(2015.0)}.
\end{remark}
Accretivity of an operator $A$ is typically associated with a corresponding notion of \emph{pseudocontractivity} for the operator $(I-A)$. In the case of uniformly accretive operators at zero, the correspondence is given as follows:
\begin{lemma}
\label{lem-pseudo}
Suppose that $A:D(A)\to 2^X$ with $0\in Aq$ is uniformly accretive at zero with modulus $\Theta_{(\cdot)}(\cdot)$. Then 
\begin{equation*}
\begin{aligned}
&\forall \varepsilon,K>0\; \forall (x,u)\in (I-A)\\
&(\norm{x-q}\in [\varepsilon,K]\to \exists j\in J(x-q)\; (\pair{u-q,j}\leq \norm{x-q}^2-\Theta_K(\varepsilon))).
\end{aligned}
\end{equation*}

\end{lemma}

\begin{proof}
If $u\in (I-A)x$ then $u=x-\bar{u}$ for $\bar{u}\in Ax$, and thus if $\norm{x-q}\in[\varepsilon,K]$ there exists some $j\in J(x-q)$ such that $\pair{\bar{u},j}\geq \Theta_K(\varepsilon)$. Therefore
\begin{equation*}
\pair{u-q,j}=\pair{x-q,j}-\pair{\bar{u},j}=\norm{x-q}^2-\pair{\bar{u},j}\leq \norm{x-q}^2-\Theta_K(\varepsilon).
\end{equation*}
\end{proof}
%
\section{An abstract technical lemma}
\label{sec-tech}

We begin by presenting an abstract quantitative lemma, which forms the main unifying scheme of the paper. This technical lemma captures a key \emph{combinatorial} idea which is shared by numerous proofs of strong convergence theorems involving accretive operators, and as we will see, quantitative versions of those theorems can be obtained in an entirely modular fashion by instantiating the parameters of our lemma in a suitable way. What is particularly interesting is that in each case we study, those instantiations are obtained by appealing to quantitative versions of assumptions which are seemingly unrelated, which here include properties imposed on the operator $A$ (Sections \ref{sec-arsii} and \ref{sec-mn}) or alternatively attributes of the underlying space $X$ (Section \ref{sec-lin}). Moreover, our abstract result applies to different approximating schemes, including implicit schemes (Sections \ref{sec-impi} and \ref{sec-arsii}) in addition to Ishikawa-type methods (Sections \ref{sec-mn} and \ref{sec-lin}).

\subsection{Rates of convergence and divergence}
\label{sec-tech-rates}

We begin by specifying quantitative versions of a couple of fundamental notions.
\begin{definition}
\label{def-conv}
Let $(\alpha_n)$ be a sequence of nonnegative reals such that $\alpha_n\to 0$. A \emph{rate of convergence} for $(\alpha_n)$ is a function $\phi:(0,\infty)\to \NN$ such that
\begin{equation*}
\forall\varepsilon>0\; \forall n\geq \phi(\varepsilon)\; (\alpha_n\leq \varepsilon).
\end{equation*}
\end{definition}
\begin{definition}
\label{def-div}
Let $(\alpha_n)$ be a sequence of nonnegative reals such that $\sum_{i=0}^\infty\alpha_i=\infty$. A \emph{rate of divergence} for $\sum_{i=0}^\infty\alpha_i$ is a function $r:\NN\times (0,\infty)\to \NN$ such that
\begin{equation*}
\forall N\in\NN\; \forall x>0\; (\sum_{i=N}^{r(N,x)}\alpha_i\geq x).
\end{equation*}
We use the convention that $\sum^m_{i=N}\alpha_i=0$ if $m<N$ and so we always 
have that $r(N,x)\ge N.$
\end{definition}
\begin{remark}
The quantitative formulation of divergence above is also used by the first author in \cite{Kohlenbach(2001.0)}. Note that a more traditional rate of divergence would be a function $f:(0,\infty)\to\NN$ satisfying 
\begin{equation*}
\forall x>0\, (\sum_{i=0}^{f(x)}\alpha_i\geq x),
\end{equation*}
which can be converted into a rate of divergence in our sense by setting $r(N,x):=f(x+S(N))$ where $S:\NN\to (0,\infty)$ is any function satisfying $\sum_{i=0}^{n-1} \alpha_i\leq S(n)$, since then we have
\begin{equation*}
\sum_{i=N}^{f(x+S(N))}\alpha_i=\sum_{i=0}^{f(x+S(N))}\alpha_i-\sum_{i=0}^{n-1}\alpha_i\geq (x+S(N))-S(N)=x.
\end{equation*}
In particular, if the $\alpha_i$ are bounded above by some $K$, we can simply set $r(N,x):=f(x+K\cdot N)$.
\end{remark}

\subsection{The technical lemma}
\label{sec-tech-lemma}

We now present our unifying lemma, which generalises similar abstract results in the literature, such as Lemma 2.2 of \cite{AlbReiSho(2002.0)} and Lemma 2.1 of \cite{MooNno(2001.0)}, the latter having been given a quantitative form as Lemma 1 of \cite{KoeKoh(2011.0)}.
\begin{lemma}
\label{lem-tech}
Let $(\theta_n)$ and $(\alpha_n)$ be sequences of nonnegative reals such that $\sum_{i=0}^\infty\alpha_i$ diverges, and suppose that for any $\varepsilon>0$ there exists some $\delta>0$ and $N\in\NN$ such that
\begin{equation*}
(\ast) \ \ \ \forall n\geq N(\varepsilon< \theta_{n+1}\to\theta_{n+1}\leq \theta_n-\alpha_n\cdot\delta).
\end{equation*}
Then $\theta_n\to 0$ as $n\to\infty$. Moreover, if:
\begin{enumerate}[(i)]

\item\label{item-techi} $K\in (0,\infty)$ satisfies $\theta_n< K$ for all $n\in\NN$,

\item\label{item-techii} $r:\NN\times (0,\infty)\to \NN$ is a rate of divergence for $\sum_{i=0}^\infty\alpha_i$,

\item\label{item-techiii} $N:(0,\infty)\to \NN$ and $\varphi:(0,\infty)\to (0,\infty)$ witness property $(\ast)$ in the sense that for all $\varepsilon> 0$ we have
\begin{equation*}
\forall n\geq N(\varepsilon)(\varepsilon< \theta_{n+1}\to \theta_{n+1}\leq \theta_n-\alpha_n\cdot\varphi(\varepsilon)),
\end{equation*}

\end{enumerate}
then $\Psi_{K,r,N,\varphi}(\varepsilon):=r(N(\varepsilon),K/\varphi(\varepsilon))+1$ is a rate of convergence for $(\theta_n)$.
\end{lemma}

\begin{proof}
We first observe that for any $\varepsilon>0$ and $n\geq N(\varepsilon)$ we have
\begin{equation*}
\theta_n\leq \varepsilon\to \theta_{n+1}\leq \varepsilon.
\end{equation*}
Otherwise, if there were some $n\geq N(\varepsilon)$ with $\theta_n\leq\varepsilon$ and $\varepsilon<\theta_{n+1}$ we would have
\begin{equation*}
\varepsilon<\theta_{n+1}\leq \theta_n-\alpha_n\cdot\varphi(\varepsilon)\leq \theta_n\leq\varepsilon.
\end{equation*}
Therefore to establish $\theta_n\to 0$ it suffices to find, for each $\varepsilon>0$, a single $n\in\NN$ with $\theta_n\leq \varepsilon$. Fixing some $\varepsilon>0$ and $j\geq N(\varepsilon)$, suppose that $\theta_{n+1}>\varepsilon$ for all $n\in\NN$ with $N(\varepsilon)\leq n\leq j.$ Then in 
particular we would have
\begin{equation*}
\alpha_n\cdot\varphi(\varepsilon)\leq \theta_n-\theta_{n+1}
\end{equation*}
for all $n$ in this range, and thus
\begin{equation*}
\varphi(\varepsilon)\sum_{n=N(\varepsilon)}^{j}\alpha_n\leq \sum_{n=N(\varepsilon)}^j(\theta_n-\theta_{n+1})=\theta_{N(\varepsilon)}-\theta_{j+1}\leq \theta_{N(\varepsilon)}< K.
\end{equation*}
But this is a contradiction for $j:=r(N(\varepsilon),K/\varphi(\varepsilon))$, and thus $\theta_{n}\leq\varepsilon$ for some $n\leq j+1$, which means that for $m\geq j+1\geq n$ we also have $\theta_m\leq \varepsilon$. 
\end{proof}

\begin{remark}
  Condition (\ref{item-techi}) of Lemma \ref{lem-tech} is not strictly necessary, as boundedness of $(\theta_n)$ is not necessary to establish $\theta_n\to 0$. However, as a rate of convergence we would then obtain e.g. $\Psi_{(\theta_n),r,N,\varphi}:=r(N(\varepsilon),(\theta_{N(\varepsilon)}+1)/\varphi(\varepsilon))$ which is dependent on the $(\theta_n)$ (or more generally, some sequence 
$(K_n)$ of upper bounds $K_n\ge\theta_n$). 
In each subsequent application of this result, we are able to supply a uniform bound $K$ for our sequence $(\theta_n)$, in which case our lemma results in a rate of convergence which is independent of the $(\theta_n)$.
\end{remark}
\begin{remark}[Linear convergence]
\label{rem-lin}
A finer analysis of Lemma \ref{lem-tech} in special cases can yield more precise convergence speeds for $(\theta_n)$. For example, suppose that $\alpha_n\geq\alpha>0$ for all $n\in\NN$ and some $\alpha$, so that a rate of divergence for $\sum_{i=0}^\infty\alpha_i$ is given by $r(N,x)=\ceil{\frac{x}{\alpha}}+N$, and suppose in addition that $N(\varepsilon)=0$ and $\varphi(\varepsilon)=c\varepsilon$ for some $c>0$ and for all $\varepsilon >0$, 
so that condition (\ref{item-techiii}) can be reduced to
\begin{equation*}
\forall n(\theta_{n+1}\leq \theta_n-\alpha c\theta_{n+1}).
\end{equation*}
Then it follows directly that
\begin{equation*}
\theta_n\leq K\left(\frac{1}{1+\alpha c}\right)^n 
\end{equation*}
and so $\theta_n\to 0$ with linear convergence speed, where a rate of convergence in our sense would be given by $\phi(\varepsilon)=\log_{1+\alpha c}(K/\varepsilon)$. This is a strict improvement of the rate of convergence suggested by Lemma \ref{lem-tech} i.e. $\Psi(\varepsilon)=\ceil{K/\alpha c\varepsilon}+1$.
\end{remark}
We conclude this section by observing that we can reformulate Lemma \ref{lem-tech} so that it no longer makes direct reference to a rate of divergence for $\sum_{i=0}^\infty\alpha_i$, but rather uses the divergence of $\sum_{i=0}^\infty\alpha_i$ \emph{implicitly}. This will later allow us to connect our quantitative convergence theorems to the numerical results presented in \cite{AlbReiSho(2002.0)}.
\begin{lemma}
\label{lem-techrate}
Let $(\theta_n)$, $(\alpha_n)$, $K$, $N$ and $\varphi$ be as in Lemma \ref{lem-tech}, and assume in addition that $\alpha_n>0$ for all $n\in\NN$. Suppose that $f:(0,\infty)\to (0,\infty)$ is strictly decreasing and continuous with $f(\varepsilon)\to \infty$ as $\varepsilon\to 0,$ and
\begin{equation*}
f(\varepsilon)\geq \sum_{i=0}^{N(\varepsilon)-1}\alpha_i+\frac{K}{\varphi(\varepsilon)}
\end{equation*}
for all $\varepsilon>0$. Then for sufficiently large $n\in\NN$ we have
\begin{equation*}
\theta_n\leq f^{-1}\left(\sum^{n-1}_{i=0}\alpha_i\right).
\end{equation*}

\end{lemma}

\begin{proof}
First note that $f$ must have an inverse $f^{-1}:(a,\infty)\to (0,\infty)$ for $a:=\inf\{f(x)\; : \; x\in (0,\infty)\}$. Define $n_0\in\NN$ to be the least 
natural number 
such that $\sum_{i=0}^{n_0}\alpha_i\in (a,\infty)$, and for $n\geq n_0+1$ define
\begin{equation*}
\varepsilon_n:=f^{-1}\left(\sum_{i=0}^{n-1}\alpha_i\right).
\end{equation*}
Applying Lemma \ref{lem-tech} for $r(N,x)$ defined to be the least $j\geq N$ such that $\sum^j_{i=N}\alpha_i\geq x$, we have $\theta_m\leq \varepsilon_n$ for all $m\geq j+1$, where $j\geq N(\varepsilon_n)$ is the least 
natural number such that
\begin{equation*}
\sum_{i=N(\varepsilon_n)}^j\alpha_i\geq \frac{K}{\varphi(\varepsilon_n)}.
\end{equation*}
Now observing that
\begin{equation*}
\sum_{i=0}^{n-1}\alpha_i=f(\varepsilon_n)\geq \sum^{N(\varepsilon_n)-1}_{i=0}\alpha_i+\frac{K}{\varphi(\varepsilon_n)}\mbox{ \ \ and thus \ \ }\sum^{n-1}_{i=N(\varepsilon_n)} \alpha_i \geq \frac{K}{\varphi(\varepsilon_n)}
\end{equation*}
it follows that $n-1\geq j$ and therefore $n\geq j+1$, which means that $\theta_n\leq \varepsilon_n$. Thus the lemma holds for all $n\geq n_0+1$.
\end{proof}

\subsection{Outline of the remainder of the paper}
\label{sec-tech-outline}

We now turn our attention towards concrete convergence theorems involving strongly accretive operators $A$. We focus on a series of examples, where in each case we utilise Lemma \ref{lem-tech} together with a modulus of uniform accretivity at zero for $A$ to carry out a quantitative analysis of the proof in question, resulting in a series of new, quantitative convergence results which each fall underneath the same unifying scheme.

\section{A simple implicit scheme}
\label{sec-impi}

Our first result will be a quantitative analysis of the following theorem of Alber et al. \cite{AlbReiSho(2002.0)}, which is based on a straightforward implicit approximation method generated by a uniformly accretive operator.
\begin{theorem}[Theorem 2.1 of \cite{AlbReiSho(2002.0)}]
\label{thm-arsi}
Let $D$ be a closed subset of $X$ and $A:D\to 2^X$ a $\psi$-strongly accretive operator for some strictly increasing $\psi$, which satisfies the range condition $(RC)$:
\begin{equation*}
D\subset (I+rA)(D), \ \ \ \forall r>0.
\end{equation*}
Then the following assertions hold:
\begin{enumerate}[(a)]

\item There exists a unique $q\in D$ such that $0\in Aq$.

\item If $(\alpha_i)$ is a sequence of positive reals with $\sum_{i=0}^\infty\alpha_i=\infty$, then if the sequence $(x_n)$ starting from some $x_0\in D$ 
satisfies 
\begin{equation*}
x_{n+1}=x_n-\alpha_{n}u_n, \ \ \ u_n\in Ax_{n+1}
\end{equation*}
we have $\norm{x_n-q}\to 0$.

\end{enumerate}

\end{theorem}
Our quantitative analysis in this case and in all those that follow will focus on the extraction of an explicit rate of convergence for $\norm{x_n-q}\to 0$ from the corresponding proof of this fact. In doing so, we adopt the following pattern:
\begin{quote}
We \emph{assume} from the outset the existence of some $q$ satisfying $0\in Aq$, and take some \emph{arbitrary} sequence $(x_n)$ satisfying the relevant approximation scheme.
\end{quote}
By focusing exclusively on the proof that $\norm{x_n-q}\to 0$, we are typically able to weaken certain conditions of the original theorem, which are often needed only to establish the \emph{existence} of a zero $0\in Aq$ or to ensure that the sequence of approximations $(x_n)$ is well-defined. As such, we obtain a rate of convergence for $\norm{x_n-q}\to 0$ which is valid in a much more general setting. At the same time, the original results guarantee us that there is always a natural context in which a zero and a corresponding approximation sequence do indeed exist!

In the case of Theorem \ref{thm-arsi} above, for the purpose of our quantitative convergence result, we are able to dispense with the range condition together with the assumption that $D$ is closed, and can take $A$ to be an arbitrary operator which is uniformly accretive at zero.

\begin{theorem}
\label{thm-arsi-quant}
Let $A:D(A)\to 2^X$ with $0\in Aq$ be uniformly accretive at zero with modulus $\Theta$. Let $(\alpha_i)$ be a sequence of nonnegative reals such that $\sum_{i=0}^\infty\alpha_i=\infty$ with modulus of divergence $r$, and suppose that $(x_n)$ and $(u_n)$ are sequences satisfying $x_n\in D(A)$ and
\begin{equation*}
x_{n+1}=x_n-\alpha_nu_n, \ \ \ u_n\in Ax_{n+1}
\end{equation*}
for all $n\in\NN$. Finally, let $K>0$ be such that $\norm{x_0-q}<K$. Then $\norm{x_n-q}\to 0$ with rate of convergence
\begin{equation*}
\Phi_{\Theta,r,K}(\varepsilon):=r(0,K^2/\Theta_K(\varepsilon))+1.
\end{equation*} 
\end{theorem}

\begin{proof}
We first observe that for any $j\in J(x_{n+1}-q)$ we have
\begin{equation}
\label{eqn-abs0}
\begin{aligned}
\norm{x_{n+1}-q}^2=&\pair{x_{n+1}-q,j}\\
=&\pair{x_n-q,j}-\alpha_n\pair{u_n,j}\\
\leq & \norm{x_n-q}\cdot \norm{x_{n+1}-q}-\alpha_n\pair{u_n,j}.
\end{aligned}
\end{equation}
We argue by induction that $\norm{x_n-q}<K$ for all $n\in\NN$: For $n=0$ this holds by assumption, while the induction step follows directly from (\ref{eqn-abs0}) together with the accretivity of $A$, which ensures that $\pair{u_n,j}\geq 0$ for some $j\in J(x_{n+1}-q)$.

Now, fixing some $n\in\NN$ and $\varepsilon>0$, suppose that $\varepsilon<\norm{x_{n+1}-q}$. Then since in particular we would have $\norm{x_{n+1}-q}\in [\varepsilon,K]$, by uniform accretivity of $A$ at zero it follows that there exists some $j\in J(x_{n+1}-q)$ such that $\pair{u_n,j}\geq \Theta_K(\varepsilon)$. Substituting this into (\ref{eqn-abs0}) and dividing by $\norm{x_{n+1}-q}$ we obtain
\begin{equation}
\label{eqn-abs1}
\norm{x_{n+1}-q}\leq \norm{x_n-q}-\frac{\alpha_n\Theta_K(\varepsilon)}{\norm{x_{n+1}-q}}<\norm{x_n-q}-\frac{\alpha_n\Theta_K(\varepsilon)}{K}.
\end{equation}
We are now able to apply Lemma \ref{lem-tech} for $\theta_n:=\norm{x_n-q}$. Conditions (\ref{item-techi}) and (\ref{item-techii}) of the lemma are clearly satisfied by $K$ and $r$, while for condition (\ref{item-techiii}) we set $N(\varepsilon):=0$ and $\varphi(\varepsilon):=\Theta_K(\varepsilon)/K$, and our rate of convergence is obtained directly.
\end{proof}
\begin{remark}
\label{rem-absi}
In \cite{AlbReiSho(2002.0)}, the operator $A$ is assumed to be $\psi$-strongly accretive for some strictly increasing $\psi$. Under the additional assumption that $0\in Aq$, $A$ must then also be uniformly accretive at zero with modulus $\Theta_K(\varepsilon)=\psi(\varepsilon)\cdot\varepsilon$, since for $u\in Ax$ with $\norm{x-q}\in [\varepsilon,K]$ there is, by $\psi$-strong accretivity, some $j\in J(x-q)$ such that
\begin{equation*}
\pair{u,j}\geq \psi(\norm{x-q})\norm{x-q}\geq \psi(\varepsilon)\cdot\varepsilon=\Theta_K(\varepsilon).
\end{equation*}
However, in the case of $\psi$-strong accretivity, we can reformulate (\ref{eqn-abs1}) as
\begin{equation*}
\norm{x_{n+1}-q}\leq \norm{x_n-q}-\frac{\alpha_n\psi(\norm{x_{n+1}-q})\norm{x_{n+1}-q}}{\norm{x_{n+1}-q}}\leq \norm{x_n-q}-\alpha_n\psi(\varepsilon)
\end{equation*}
and thus an improved rate of convergence for $\norm{x_n-q}\to 0$ is given by 
\begin{equation*}
\Phi_{\psi,r,K}:=r(0,K/\psi(\varepsilon))+1.
\end{equation*}
Moreover, following Remark \ref{rem-lin}, for the particular case that $\alpha_n\geq \alpha>0$ for some $\alpha$ and $\psi(\varepsilon)=c\varepsilon$ for some $c>0$ and for all $\varepsilon>0$, we would have
\begin{equation*}
\norm{x_n-q}\leq K\left(\frac{1}{1+\alpha c}\right)^n
\end{equation*}
and thus $\norm{x_n-q}\to 0$ linearly. This observation is analogous to the Example (1) sketched on p. 97 of \cite{AlbReiSho(2002.0)}.
\end{remark}
By appealing to Lemma \ref{lem-techrate}, we obtain an implicit rate of convergence closely related to Theorem 3.1 of \cite{AlbReiSho(2002.0)}.
\begin{corollary}
\label{cor-arsi-rate}
Let $A:D(A)\to 2^X$ with $0\in Aq$ be a $\psi$-strongly accretive operator for some strictly increasing $\psi$, and otherwise let $(\alpha_i),r,(x_n),(u_n)$ and $K$ be as in Theorem \ref{thm-arsi-quant}. Then $\norm{x_n-q}\to 0$ with
\begin{equation*}
\norm{x_n-q}<\psi^{-1}\left(\frac{K}{\sum_{i=0}^{n-1}\alpha_i}\right)
\end{equation*}
sufficiently large $n\in\NN$.
\end{corollary}

\begin{proof}
We apply Lemma \ref{lem-techrate} with parameters instantiated as in Remark \ref{rem-absi} i.e. $N(\varepsilon)=0$ and $\varphi(\varepsilon)=\psi(\varepsilon)$. Then in particular, we can define our bounding function $f:(0,\infty)\to (0,\infty)$ by $f(\varepsilon):=K/\psi(\varepsilon)$, observing that as $\varepsilon\to 0$ then $\psi(\varepsilon)\to 0$ and hence $f(\varepsilon)\to \infty$. The result then follows by observing that $f^{-1}(x)=\psi^{-1}(K/x)$.
\end{proof}

\begin{remark}
Note that Corollary \ref{cor-arsi-rate} is broadly analogous but not identical to the corresponding Theorem 3.1 of \cite{AlbReiSho(2002.0)}, which is to be expected, since the latter uses an integral comparison rather than a rate of divergence for $\sum_{i=0}^\infty\alpha_i$.
\end{remark}

\section{An implicit scheme using approximating operators}
\label{sec-arsii}

Our second case study is also taken from \cite{AlbReiSho(2002.0)}. Here, the implicit scheme studied in the previous section is modified to one of the form
\begin{equation*}
x_{n+1}=x_n-\alpha_nu_n, \ \ \ u_n\in A_nx_{n+1}
\end{equation*}
for some sequence of operators $(A_n)$, where in order to maintain convergence of the $(x_n)$ to some zero $q$ when it exists, a convergence property for the $(A_n)$ is required. In \cite{AlbReiSho(2002.0)} this takes the form of approximation relative to the Hausdorff distance.
\begin{definition}
\label{def-app}
Let $A$ and $A_n:D\to 2^X$ be operators defined on some subset $D$ of $X$ for $n\in\NN$. We say that the sequence $(A_n)$ approximates the operator $A$ if there exists a sequence of positive reals $(h_n)$ with $h_n\to 0$ as $n\to\infty$ such that 
\begin{equation*}
\forall x\in D\; \forall n\in\NN\; (H(A_nx,Ax)\leq h_n\xi(\norm{x}))
\end{equation*}
where $\xi:[0,\infty)\to [0,\infty)$ is some given function and $H$ denotes the Hausdorff distance between sets, defined as usual by
\begin{equation*}
H(P,Q)=\max\{\sup_{x\in X}\inf_{y\in Y}\norm{x-y},\sup_{y\in Y}\inf_{x\in X}\norm{x-y}\}.
\end{equation*}
\end{definition}
We analyse the following generalisation of Theorem 2.1 of \cite{AlbReiSho(2002.0)}.
\begin{theorem}[Theorem 4.1 of \cite{AlbReiSho(2002.0)}]
\label{thm-arsii}
Let $D$ be a closed subset of $X$ and $A:D\to 2^X$ a $\psi$-strongly accretive operator for some strictly increasing $\psi$ which satisfies the range condition (RC). Suppose that the sequence of operators $A_n:D\to 2^X$ approximates $A$ and each $A_n$ satisfies the strong range condition that for any $r>0$ and $u\in D$ there exists a unique $x\in D$ with
\begin{equation*}
u\in (I+rA_n)x.
\end{equation*}
If $(\alpha_i)$ is a sequence of positive reals with $\sum_{i=0}^\infty \alpha_i=\infty$, $(x_n)$ is the sequence starting from some $x_0\in D$ and defined by 
\begin{equation*}
x_{n+1}=x_n-\alpha_nu_n,\ \ \ u_n\in A_nx_{n+1},
\end{equation*}
and $(x_n)$ is bounded, then $\norm{x_n-q}\to 0$ for the unique zero $0\in Aq$ (which exists by Theorem \ref{thm-arsi} of this paper i.e. Theorem 2.1 of \cite{AlbReiSho(2002.0)}).
\end{theorem}
As in the previous section, we simply assume the existence of some $0\in Aq$, and as such, omit all range conditions from our quantitative version of this result. However, that $(A_n)$ approximates the operator $A$ is essential for convergence of the $(x_n)$. Just as various notions of strong accretivity are replaced by uniform accretivity at zero, we define below a general, uniform variant of the approximation property which reflects the more restricted way in which that property is actually used in the proof of Theorem \ref{thm-arsii}.
\begin{definition}
\label{def-haus}
We define the predicate $H^\ast\subseteq 2^X\times 2^X\times (0,\infty)$ by
\begin{equation*}
H^\ast[P,Q,a]:\equiv \forall u\in P\; \exists v\in Q\; (\norm{u-v}\leq a).
\end{equation*}
\end{definition}

\begin{definition}
\label{def-unifapp}
Let $A$ and $A_n:D\to 2^X$ be operators with $D(A)=D(A_n)=D.$ 
We say that $(A_n)$ \emph{uniformly approximates} $A$ with the {\it rate $\mu_{(\cdot)}(\cdot):(0,\infty)\times (0,\infty)\to \NN$ of uniform approximation} if 
\begin{equation*}
\begin{aligned}
\forall K,\varepsilon >0\; \forall n\ge \mu_K(\varepsilon)\,\forall x\in D
\left(\norm{x}\leq K\to H^\ast[A_nx,Ax,\varepsilon]\right).
\end{aligned}
\end{equation*}
\end{definition}

\begin{lemma}
\label{lem-hausapprox}
Suppose that $(A_n)$ approximates $A$ (in the sense of Definition \ref{def-app}) with respect to $(h_n)$ and some $\xi:[0,\infty)\to [0,\infty)$, and moreover there is a function $\xi^\ast:(0,\infty)\to (0,\infty)$ satisfying
\begin{equation*}
\forall x,y\in [0,\infty)(x\leq y\wedge y>0 \to \xi(x)\leq \xi^\ast(y)).
\end{equation*}
Then $(A_n)$ uniformly approximates $A$. Moreover, if $\phi:(0,\infty)\to\NN$ is a rate of convergence for $h_n\to 0$ then
\begin{equation*}
\mu_L(\varepsilon):=\phi(2\varepsilon/3\xi^\ast(L))
\end{equation*}
is a rate of uniform approximation for $(A_n)$ and $A$.
\end{lemma}

\begin{proof}
We first observe that for $P,Q\in 2^X$, whenever $H(P,Q)<a$ for some $a\in (0,\infty)$ it follows that $H^\ast[P,Q,a]$. To see this, note that $H(P,Q)<a$ implies in particular that for all $u\in P$ we have
\begin{equation*}
\inf_{v\in V}\norm{u-v}<a
\end{equation*}
and thus there must exist some $v\in V$ with $\norm{u-v}< a$. Now, fixing some $n,K$ (with $K>0$) and $x\in D$ with $\norm{x}\leq K$, we have
\begin{equation*}
H(A_nx,Ax)\leq h_n\xi(\norm{x})\leq h_n\xi^\ast(K)<\tfrac{3}{2} h_n\xi^\ast(K)
\end{equation*}
where for the last step we use $\xi^\ast(K)>0$. Now let $n\ge 
\phi(2\varepsilon/3\xi^*(K)),$ then $\tfrac{3}{2} h_n\xi^\ast(K)
\le \varepsilon$ and so $H^*(A_nx,Ax,\varepsilon).$
\end{proof}
We are now ready to state and prove our quantitative formulation of Theorem \ref{thm-arsii}.
\begin{theorem}
\label{thm-arsii-quant}
Let $A:D\to 2^X$ with $0\in Aq$ be uniformly accretive at zero with modulus $\Theta$, and $A_n:D\to 2^X$ be a sequence of operators which uniformly approximates $A$ with rate $\mu$. Let $(\alpha_i)$ be a sequence of nonnegative reals such that $\sum_{i=0}^\infty\alpha_i=\infty$ with modulus of divergence $r$, and suppose that $(x_n)$ and $(u_n)$ are sequences satisfying $x_n\in D$ and 
\begin{equation*}
x_{n+1}=x_n-\alpha_nu_n, \ \ \ u_n\in A_nx_{n+1}
\end{equation*}
for all $n\in\NN$. Finally, suppose $K,K'\in (0,\infty)$ satisfy $\norm{x_n-q}<K$ for all $n\in\NN$ and $\norm{q}<K'$. Then $\norm{x_n-q}\to 0$ with rate of convergence
\begin{equation*}
\Phi_{\Theta,\mu,r,K,K'}(\varepsilon):=r\left(\mu_{K+K'}\left(\frac{\Theta_{K}(\varepsilon)}{2K}\right),\frac{K^2}{\Theta_K(\varepsilon)}\right)+1.
\end{equation*}
\end{theorem}
\begin{proof}
Using the assumption that $(A_n)$ uniformly approximates $A$, together with the assumption that $x_{n+1}\in D$ and $\norm{x_{n+1}}\leq \norm{x_{n+1}-q}+\norm{q}<K+K'$ for each $n\in\NN$, we have $H^\ast[A_nx_{n+1},Ax_{n+1},\Theta_K
(\sqrt{\varepsilon})/2K]$ for all $n\ge N(\varepsilon):=
\mu_{K+K'}\left(\Theta_K(\sqrt{\varepsilon})/2K\right).$ 
In particular, this means that for all $n\ge N(\varepsilon)$ there exists some $v_n\in Ax_{n+1}$ such that 
\begin{equation}
\label{eqn-abss0}
\norm{u_n-v_n}\le \tfrac{\Theta_K(\sqrt{\varepsilon})}{2K}.
\end{equation}
Now, for any $j\in J(x_{n+1}-q)$ we have
\begin{equation}
\label{eqn-abss1}
\begin{aligned}
\norm{x_{n+1}-q}^2=&\norm{x_n-\alpha_nu_n-q}^2\\
\stackrel{L. \ref{lem-subdif}}{\leq}&\norm{x_n-q}^2- 2\alpha_n\pair{u_n,j}\\
=&\norm{x_n-q}^2+2\alpha_n\pair{v_n-u_n,j}-2\alpha_n\pair{v_n,j}\\
\leq &\norm{x_n-q}^2+2\alpha_n\norm{v_n-u_n}\cdot \norm{x_{n+1}-q}-2\alpha_n\pair{v_n,j}\\
\leq &\norm{x_n-q}^2-2\alpha_n(\pair{v_n,j}- 
\tfrac{\Theta_K(\sqrt{\varepsilon})}{2}),
\end{aligned}
\end{equation}
where for the last step we use (\ref{eqn-abss0}) by which
\begin{equation*}
\norm{v_n-u_n}\cdot\norm{x_{n+1}-q}< K\cdot 
\tfrac{\Theta_K(\sqrt{\varepsilon})}{2K}.
\end{equation*}
Now suppose that $\varepsilon<\norm{x_{n+1}-q}^2$, and thus $\norm{x_{n+1}-q}\in [\sqrt{\varepsilon},K]$. Then by uniform accretivity of $A$ at zero there exists some $j\in J(x_{n+1}-q)$ such that $\pair{v_n,j}\geq \Theta_{K}(\sqrt{\varepsilon})$ and hence
\begin{equation}
\label{eqn-abss2}
\begin{aligned}
\pair{v_n,j}-\tfrac{\Theta_K(\sqrt{\varepsilon})}{2}&\geq \Theta_{K}(\sqrt{\varepsilon})-\tfrac{\Theta_K(\sqrt{\varepsilon})}{2} \\
&= \tfrac{\Theta_{K}(\sqrt{\varepsilon})}{2}.
\end{aligned}
\end{equation}
Substituting (\ref{eqn-abss2}) into (\ref{eqn-abss1}), for $n\geq N(\varepsilon)$ and $\varepsilon<\norm{x_{n+1}-q}^2$ we have
\begin{equation*}
\norm{x_{n+1}-q}^2\leq \norm{x_n-q}^2-\alpha_n\cdot \varphi(\varepsilon)
\end{equation*}
for $\varphi(\varepsilon):=\Theta_K(\sqrt{\varepsilon})$. Therefore applying Lemma \ref{lem-tech} for $\theta_n:=\norm{x_n-q}^2\le K^2$, where condition (\ref{item-techi}) is witnessed by $K^2$, (\ref{item-techii}) by $r$ and (\ref{item-techiii}) by $N$ and $\varphi$ as defined above, we obtain a rate of convergence for $\norm{x_n-q}^2\to 0$, which can be modified to a rate of convergence for $\norm{x_n-q}\to 0$ by substituting $\varepsilon^2$ for $\varepsilon$. 
\end{proof}
We conclude our study of \cite{AlbReiSho(2002.0)} with a final quantitative result that forms a more direct counterpart of  Theorem 4.1 in 
\cite{AlbReiSho(2002.0)}, which brings together Lemma \ref{lem-hausapprox} and Theorem \ref{thm-arsii-quant} above, and in addition incorporates the discussion on pp.100-101 of \cite{AlbReiSho(2002.0)}, in which boundedness of the $\norm{x_n}$ is replaced by the a priori condition that the $A_n$ are each accretive and $\sum_{n=0}^\infty \alpha_n h_n<\infty$.
\begin{theorem}
\label{thm-arsiii-quant}
Let $A:D\to 2^X$ with $0\in Aq$ be uniformly accretive at zero with modulus $\Theta$, and $A_n:D\to 2^X$ be a sequence of accretive operators each satisfying the range condition (RC) which approximates $A$ with respect to $(h_n)$ and some $\xi:[0,\infty)\to [0,\infty)$. Let $\phi:(0,\infty)\to\NN$ be a rate of convergence for $h_n\to 0$ and $\xi^\ast:(0,\infty)\to (0,\infty)$ a function satisfying
\begin{equation*}
\forall x,y\in [0,\infty)(x\leq y \wedge y>0\to \xi(x)\leq \xi^\ast(y)).
\end{equation*}
In addition, let $(\alpha_i)$ be a sequence of nonnegative reals such that $\sum_{i=0}^\infty\alpha_i=\infty$ with modulus of divergence $r$ and 
$\sum^{\infty}_{i=0}\alpha_ih_i<\infty$, and suppose that $(x_n)$ and $(u_n)$ are sequences satisfying $x_n\in D$ and
\begin{equation*}
x_{n+1}=x_n-\alpha_nu_n, \ \ \ u_n\in A_{n}x_{n+1}
\end{equation*}
for all $n\in\NN$. Finally, suppose that $K_0,K_1,K_2\in (0,\infty)$ satisfy $\norm{x_0-q}<K_0$, $\norm{q}<K_1$ and $\sum_{i=0}^n\alpha_ih_i<K_2$ for all $n\in\NN$. Then $\norm{x_n-q}\to 0$ with rate of convergence
\begin{equation*}
\Phi_{\Theta,\phi,\xi^\ast,r,K_0,K_1,K_2}(\varepsilon):=r\left(\phi\left(\frac{\Theta_K(\varepsilon)}{3K\cdot \xi^\ast(K+K_1)}\right),\frac{K^2}{\Theta_K(\varepsilon)}\right)+1
\end{equation*}
for $K:=K_0+K_2\cdot \xi^\ast(K_1)$.
\end{theorem}
\begin{proof}
Since $A_n$ satisfies the range condition, we have $q\in (I+\alpha_nA_n)(D)$, which means there exist a pair of sequences $(y_n)$ and $(v_n)$ with
\begin{equation*}
q=y_n+\alpha_nv_n, \ \ \ v_n\in A_ny_n
\end{equation*}
for all $n\in\NN$. We now observe that since $\norm{q}<K_1$ we have $H(A_nq,Aq)<h_n\xi^\ast(K_1)$, and thus since $0\in Aq$ there exists some $w_n\in A_nq$ satisfying $\norm{w_n}< h_n\xi^\ast(K_1)$. Now for any $j\in J(y_n-q)$ we have
\begin{equation}
\label{eqn-arsss0}
\norm{y_n-q}^2=\pair{y_n-q,j}=-\alpha_n\pair{v_n,j}=-\alpha_n\pair{w_n,j}-\alpha_n\pair{v_n-w_n,j}.
\end{equation}
Since $A_n$ is accretive there exists some $j\in J(y_n-q)$ such that $\pair{v_n-w_n,j}\geq 0$ and substituting this into (\ref{eqn-arsss0}) we get
\begin{equation*}
\norm{y_n-q}^2\leq -\alpha_n\pair{w_n,j}\leq \alpha_n\norm{w_n}\norm{y_n-q}
\end{equation*}
and therefore
\begin{equation}
\label{eqn-arsss1}
\norm{y_n-q}\leq \alpha_n\norm{w_n}\leq \alpha_nh_n\xi^\ast(K_1).
\end{equation}
By a similar calculation we see that for $j\in J(x_{n+1}-y_n)$ we have
\begin{equation*}
\norm{x_{n+1}-y_n}^2=\pair{x_{n+1}-y_n,j}=\pair{x_n-q,j}-\alpha_n\pair{u_n-v_n,j}
\end{equation*}
and again by accretivity of $A_n$ on $u_n\in A_nx_{n+1}$ and $v_n\in A_ny_n$ we see that
\begin{equation}
\label{eqn-arsss2}
\norm{x_{n+1}-y_n}\leq \norm{x_n-q}.
\end{equation}
Putting (\ref{eqn-arsss1}) and (\ref{eqn-arsss2}) together we obtain
\begin{equation*}
\norm{x_{n+1}-q}\leq \norm{x_{n+1}-y_n}+\norm{y_n-q}\leq \norm{x_n-q}+\alpha_nh_n\xi^\ast(K_1)
\end{equation*}
and therefore
\begin{equation*}
\norm{x_n-q}\leq \norm{x_0-q}+\sum_{i=0}^{n-1}\alpha_ih_i\xi^\ast(K_1)<K_0+\xi^\ast(K_1)\cdot K_2.
\end{equation*}
This establishes boundedness of $\norm{x_n-q}$ for $n\in\NN$. We can now apply Theorem \ref{thm-arsii-quant} for $K:=K_0+\xi^\ast(K_1)\cdot K_2$, $K':=K_1$ and (by Lemma \ref{lem-hausapprox}) $\mu_L(\varepsilon):=\phi(2\varepsilon/3\xi^\ast(L))$ to obtain the given rate of convergence.
\end{proof}

\section{An Ishikawa-type scheme for uniformly continuous operators}
\label{sec-mn}

Our next result is a quantitative analysis of a theorem due to Moore and Nnoli, which rather than the implicit schemes studied in the previous section deals with an explicit Ishikawa-type method for approximating zeros of accretive operators $A$. Here, convergence is made possible by demanding that the operator $A$ be uniformly continuous in the following sense.
\begin{definition}
\label{def-unifcont}
Let $CB(X)$ denote the family of all nonempty subsets of $X$ which are closed and bounded.
An operator $A:D(A)\to CB(X)\subset 2^X$ is said to be uniformly continuous if it satisfies
\begin{equation*}
\label{eqn-unifcont}
\forall \varepsilon>0\; \exists \delta>0\; \forall x,y\in X (\norm{x-y}\leq \delta\to H(Ax,Ay)\leq \varepsilon),
\end{equation*}
where we recall that $H$ denotes the Hausdorff distance.
\end{definition}
\begin{theorem}[Theorem 2.2 of \cite{MooNno(2001.0)}]
Let $A:D(A)\to CB(X)$ be a uniformly continuous and uniformly quasi-accretive operator with nonempty closed values such that the range of $(I-A)$ is bounded and $0\in Aq$ for some $q\in X$. Let $(\alpha_n),(\beta_n)$ be sequences in $[0,\tfrac{1}{2})$ such that $\alpha_n\to 0$, $\beta_n\to 0$ and $\sum_{i=0}^\infty\alpha_i=\infty$. Finally, let $(x_n)$ be the sequences generated from some $x_0\in X$ satisfying the Ishikawa-type scheme
\begin{equation*}
\begin{aligned}
x_{n+1}&=(1-\alpha_n)x_n+\alpha_nu_n, \ \ \ u_n\in (I-A)y_n\\
y_n&=(1-\beta_n)x_n+\beta_nv_n, \ \ \ v_n\in (I-A)x_n
\end{aligned}
\end{equation*}
Then $(x_n)$ converges strongly to $q$.
\end{theorem}
\begin{remark}
The notion of uniform quasi-accretivity (cf. \cite[Definition 1.3]{MooNno(2001.0)}) is essentially a formulation of uniform $\phi$-accretivity for zeros, and will in any case be replaced by our notion of uniform accretivity at zero.
\end{remark}
We now present our computational interpretation of the above theorem, which in particular replaces uniform continuity of $A$ with the following quantitative condition involving the Hausdorff-like predicate $H^\ast$ introduced in Definition \ref{def-haus}.
\begin{definition}
\label{def-mod-unifcont}
Let $A:D(A)\to 2^X$ be an operator. A function $\varpi:(0,\infty)\to (0,\infty)$ is called a \emph{modulus of uniform continuity} for $A$ if it satisfies
\begin{equation*}
\forall \varepsilon>0\; \forall x,y\in D(A)(\norm{x-y}\leq \varpi(\varepsilon)\to H^\ast[Ax,Ay,\varepsilon]).
\end{equation*}
\end{definition}
\begin{remark}
Note that given some $A:D(A)\to CB(X)$, if $\omega:(0,\infty)\to (0,\infty)$ is a traditional modulus of uniform continuity with respect to the Hausdorff metric, in that it satisfies
\begin{equation}
\label{eqn-modcont}
\forall \varepsilon>0\; \forall x,y\in D(A) (\norm{x-y}\leq \omega(\varepsilon)\to H(Ax,Ay)\leq \varepsilon),
\end{equation}
then $\varpi(\varepsilon):=\omega(\varepsilon/2)$ is a modulus of uniform continuity for $A$ in the sense of Definition \ref{def-mod-unifcont}. To see this, note that if $\norm{x-y}\leq \omega(\varepsilon/2)$ then $H(Ax,Ay)\leq\varepsilon/2<\varepsilon$, which implies that for any $u\in Ax$ there must exist some $v\in Ay$ with $\norm{u-v}<\varepsilon$, which is just $H^\ast[Ax,Ay,\varepsilon]$. However, possessing a modulus of uniform continuity $\varpi$ is more general than possessing some $\omega$ satisfying (\ref{eqn-modcont}), since in particular the latter only makes sense when $H(Ax,Ay)$ always exists, whereas our formulation allows us to drop assumptions about the range of $A$, and so in particular we do not require $A$ to always return nonempty closed values.
\end{remark}
\begin{theorem}
\label{thm-mn-quant}
Let $A:D(A)\to 2^X$ with $0\in Aq$ be uniformly accretive at zero with modulus $\Theta$, and in addition suppose that $A$ has a modulus of uniform continuity $\varpi:(0,\infty)\to (0,\infty)$. Assume that $R(I-A)$ is bounded. Let $(\alpha_n),(\beta_n)$ be sequences in $[0,\tfrac{1}{2})$ such that $\alpha_n,\beta_n\to 0$ with joint rate of convergence $\phi$ and $\sum_{i=0}^\infty\alpha_i=\infty$ with rate of divergence $r$. Suppose that $(x_n),(y_n),(u_n)$ and $(v_n)$ are sequences satisfying $x_n,y_n\in D(A)$ and 
\begin{equation*}
\begin{aligned}
x_{n+1}&=(1-\alpha_n)x_n+\alpha_nu_n, \ \ \ u_n\in (I-A)y_n\\
y_n&=(1-\beta_n)x_n+\beta_nv_n, \ \ \ v_n\in (I-A)x_n
\end{aligned}
\end{equation*}
Finally, suppose that $K_0,K_1\in (0,\infty)$ satisfy $\norm{w}<K_0$ for all $w\in R(I-A)$ and $\norm{x_0-q}<K_1$. Then $\norm{x_n-q}\to 0$ as $n\to\infty$ with rate of convergence
\begin{equation*}
\begin{aligned}
&\Phi_{\Theta,\varpi,\phi,r,K_0,K_1}(\varepsilon)\\
&:=r\left(\phi\left(\min\left\{\frac{1}{4},\frac{1}{6K}\min\left\{\frac{\Theta_K(\varepsilon)}{16K},\varpi\left(\frac{\Theta_K(\varepsilon)}{16K}\right)\right\}\right\}\right),\frac{K^2}{\Theta_K(\varepsilon)}\right)+1
\end{aligned}
\end{equation*}
for $K:=2K_0+K_1$.
\end{theorem}

\begin{proof}
We first show by induction that $\norm{x_n-q}<K$ for $K=2K_0+K_1$. For the base case we have $\norm{x_0-q}<K_1<K$, and for the induction step we calculate:
\begin{equation*}
\begin{aligned}
\norm{x_{n+1}-q}&=\norm{(1-\alpha_n)(x_n-q)+\alpha_n(u_n-q)}\\
&\leq (1-\alpha_n)\norm{x_n-q}+\alpha_n\norm{u_n-q}\\
&< (1-\alpha_n)K+\alpha_n K\\
&=K
\end{aligned}
\end{equation*}
where to establish $\norm{u_n-q}<K$ we use that $q\in (I-A)q$ and hence 
$q\in R(I-A)$, and - by assumption - $u_n\in R(I-A)$, from which we see that $\norm{u_n-q}\leq \norm{u_n}+\norm{q}<2K_0<K$. 
%
%
We are now also able to show that
\begin{equation}
\label{eqn-mn0}
\begin{aligned}
\norm{y_n-x_{n+1}}&=\norm{\alpha_nx_n-\beta_nx_n+\beta_nv_n-\alpha_nu_n}\\
&\leq \alpha_n\norm{x_n}+\beta_n\norm{x_n}+\beta_n\norm{v_n}+\alpha_n\norm{u_n}\\
&\leq 3(\alpha_n+\beta_n) K,
\end{aligned}
\end{equation}
where for the last step we use that $\norm{u_n},\norm{v_n}<K_0<K$ and $\norm{x_n}\leq \norm{x_n-q}+\norm{q}<K+K_0<2K$. Appealing to the joint rate of convergence $\phi$ for $\alpha_n,\beta_n\to 0$ we see that for $\tilde{\delta}\le 
\delta/6K$
\begin{equation}
\label{eqn-mn1}
\forall n\geq \phi(\tilde{\delta})(\norm{y_n-x_{n+1}}\leq \delta).
\end{equation}
For the remainder of the proof we fix some $\varepsilon>0$ and suppose that $\varepsilon< \norm{x_{n+1}-q}^2$ and thus $\norm{x_{n+1}-q}\in [\sqrt{\varepsilon},K]$. We now suppose that
\[ (*) \ n\geq N(\varepsilon):= \phi\left(\min\left\{\tfrac{1}{4},
\tfrac{\Theta_K(\sqrt{\varepsilon})}{96K^2},\tfrac{\varpi(\Theta_K(\sqrt{\varepsilon})/16K)}{6K}\right\}\right). \]
Then by (\ref{eqn-mn1}) we have for $\delta_0:=\Theta_K(\sqrt{\varepsilon})/16K$ 
\begin{equation*}
\label{eqn-mn1.5}
\norm{y_n-x_{n+1}}\leq \delta_0\mbox{ \ \ and \ \  }H^\ast[Ay_n,Ax_{n+1},\delta_0],
\end{equation*}
where the second property follows from the fact that $\varpi$ is a modulus of uniform continuity for $A$ and in addition $\norm{y_n-x_{n+1}}\leq \varpi(\delta_0)$. Now, since $u_n\in (I-A)y_n$ we have $u_n=y_n-\lambda_n$ for some $\lambda_n\in Ay_n$, and similarly $v_n=x_n-\sigma_n$ for some $\sigma_n\in Ax_n$. But since $H^\ast[Ay_n,Ax_{n+1},\delta_0]$ there must also be some $\bar{\sigma}_{n+1}\in Ax_{n+1}$ with $\norm{\lambda_n-\bar{\sigma}_{n+1}}\leq \delta_0$, and thus setting $\bar{v}_{n+1}:=x_{n+1}-\bar{\sigma}_{n+1}\in (I-A)x_{n+1}$ we have
\begin{equation}
\label{eqn-mn2}
\norm{u_n-\bar{v}_{n+1}}=\norm{y_n-\lambda_n-x_{n+1}+\bar{\sigma}_{n+1}}\leq\norm{y_n-x_{n+1}}+\norm{\lambda_n-\bar{\sigma}_{n+1}}\leq 2\delta_0.
\end{equation}
Now, using  Lemma \ref{lem-subdif} on $x_{n+1}-q=x+y$ for $x:=(1-\alpha_n)(x_n-q)$ and $y:=\alpha_n(u_n-q)$ we see that for any $j\in J(x_{n+1}-q)$ we have
\begin{equation}
\label{eqn-mn3}
\begin{aligned}
&\norm{x_{n+1}-q}^2\\
&\leq(1-\alpha_n)^2\norm{x_n-q}^2+2\alpha_n\pair{u_n-q,j}\\
&=(1-2\alpha_n)\norm{x_n-q}^2+\alpha_n^2\norm{x_n-q}^2+2\alpha_n\pair{u_n-q,j}\\
&\leq (1-2\alpha_n)\norm{x_n-q}^2+\alpha_n(\alpha_n K^2+2\pair{u_n-q,j})\\
&= (1-2\alpha_n)\norm{x_n-q}^2+\alpha_n(\alpha_n K^2+2\pair{\bar{v}_{n+1}-q,j}+2\pair{u_n-\bar{v}_{n+1},j})\\
&\leq (1-2\alpha_n)\norm{x_n-q}^2+\alpha_n(\alpha_n K^2+2\pair{\bar{v}_{n+1}-q,j}+4\delta_0K)
\end{aligned}
\end{equation}
where for the last step we use (\ref{eqn-mn2}) to establish
\begin{equation*}
\pair{u_n-\bar{v}_{n+1},j}\leq \norm{u_n-\bar{v}_{n+1}}\cdot \norm{x_{n+1}-q}\leq 2\delta_0\cdot K.
\end{equation*}
Now, since $\norm{x_{n+1}-q}\in [\sqrt{\varepsilon},K]$, by Lemma \ref{lem-pseudo} there is some $j\in J(x_{n+1}-q)$ such that
\begin{equation*}
\pair{\bar{v}_{n+1}-q,j}\leq \norm{x_{n+1}-q}^2-\Theta_K(\sqrt{\varepsilon})
\end{equation*}
and substituting this into (\ref{eqn-mn3}) we obtain
\begin{equation*}
(1-2\alpha_n)\norm{x_{n+1}-q}^2\leq (1-2\alpha_n)\norm{x_n-q}^2+\alpha_n(\alpha_n K^2-2\Theta_K(\sqrt{\varepsilon})+4\delta_0K).
\end{equation*}
Dividing both sides by $(1-2\alpha_n)>0$ we get
\begin{equation*}
\label{eqn-mn4}
\begin{aligned}
\norm{x_{n+1}-q}^2&\leq \norm{x_n-q}^2-\left(\frac{2\alpha_n}{1-2\alpha_n}\right)\Theta_K(\sqrt{\varepsilon})+\left(\frac{\alpha_nK}{1-2\alpha_n}\right)(\alpha_n K+4\delta_0)\\
&\leq \norm{x_n-q}^2-2\alpha_n\Theta_K(\sqrt{\varepsilon})+\left(\frac{\alpha_nK}{1-2\alpha_n}\right)(\alpha_n K+4\delta_0)
\end{aligned}
\end{equation*}
and therefore
\begin{equation}
\label{eqn-mn5}
\norm{x_{n+1}-q}^2\leq \norm{x_n-q}^2-\alpha_n(2\Theta_K(\sqrt{\varepsilon})- \delta_1)
\end{equation}
for
\begin{equation*}
\delta_1:=\left(\frac{K}{1-2\alpha_n}\right)(\alpha_n K+4\delta_0).
\end{equation*}
Now $(*)$ also implies that (using that $\alpha_n\leq 1/4$ implies 
$1/(1-2\alpha_n)\leq 2$)
\begin{equation*}
\delta_1\leq 2K\left(\left(\frac{\Theta_K(\sqrt{\varepsilon})}{96K^2}\right) K+4\left(\frac{\Theta_K(\sqrt{\varepsilon})}{16K}\right)\right)<\Theta_K(\sqrt{\varepsilon})
\end{equation*}
and thus by (\ref{eqn-mn5}), under the assumption that $\varepsilon\leq\norm{x_{n+1}-q}^2$ we have shown that 
\begin{equation*}
\norm{x_{n+1}-q}^2\leq \norm{x_n-q}^2-\alpha_n\cdot\varphi(\varepsilon)
\end{equation*}
for $\varphi(\varepsilon):=\Theta_K(\sqrt{\varepsilon})$ and for all $n
\ge N(\varepsilon)$, where $N(\varepsilon)$ is defined in $(*)$. 
\\ We can now apply Lemma \ref{lem-tech} for $\theta_n:=\norm{x_n-q}^2$ and $(\alpha_n)$. Conditions (\ref{item-techi}) and (\ref{item-techii}) are satisfied for $K^2$ and $r$ respectively, and we have established condition (\ref{item-techiii}) for $\varphi$ and $N$ as defined above. The stated rate of convergence is then obtained directly from the rate of convergence for $\norm{x_n-q}^2$ given by the lemma, which as before is converted to one for $\norm{x_n-q}$ by substituting $\varepsilon^2$ for $\varepsilon$. 
\end{proof}

\section{An Ishikawa-type scheme for uniformly smooth spaces}
\label{sec-lin}

Our final application concerns another Ishikawa-type scheme, but in contrast to the previous section, uniform continuity of $A$ is now exchanged for uniform smoothness of the underlying space. This results in a somewhat different approach for establishing strong convergence, but is nevertheless still subsumed under our general framework. Convergence results pertaining to Ishikawa-type schemes based on accretive operators in uniformly smooth spaces can be found in several places in the literature. The quantitative result presented here is based on an extension of \cite[Theorem 4.2]{Chang} due to Lin \cite[Theorem 2.1]{Lin(2004.0)}, the latter involving an Ishikawa-type scheme based on two accretive operators. We first establish a quantitative version of the notion of uniform smoothness.
\begin{definition}
\label{def-unifsmooth}
A Banach space $X$ is uniformly smooth if for all $\varepsilon>0$ there exists some $\delta>0$ such that
\begin{equation*}
(\ast) \ \ \ \forall x,y\in X(\norm{x}=1\wedge \norm{y}\leq\delta\to \norm{x+y}+\norm{x-y}\leq 2+\varepsilon\norm{y}).
\end{equation*}
Any function $\tau:(0,\infty)\to (0,\infty)$ such that $\delta=\tau(\varepsilon)$ satisfies $(\ast)$ is called a \emph{modulus of uniform smoothness} for $X$.
\end{definition}
It is well-known that in uniformly smooth spaces, the normalized duality mapping $J$ is single-valued and uniformly continuous. A quantitative formulation of this fact follows directly from Proposition 2.5 of \cite{KohLeu(2012.1)} (note 
that we use here the notion of `modulus of continuity' from computable 
analysis which differs from the modulus $\alpha$ defined in the current  
context at the beginning of section 2 of \cite{Reich(1978)} which does 
not provide a rate of convergence for $\lim\limits_{t\to 0^+} \alpha(t)=0$):
\begin{lemma}[\cite{KohLeu(2012.1)}]
\label{lem-unifsmooth}
Let $X$ be uniformly smooth with modulus $\tau$. Define $\omega_\tau:(0,\infty)\times (0,\infty)\to (0,\infty)$ by
\begin{equation*}
\omega_\tau(d,\varepsilon):=\frac{\varepsilon^2}{12d}\cdot \tau\left(\frac{\varepsilon}{2d}\right), \ \ \ \varepsilon\in (0,2], d\geq 1
\end{equation*}
with $\omega_\tau(d,\varepsilon):=\omega_\tau(1,\varepsilon)$ for $d<1$ and $\omega_\tau(d,\varepsilon):=\omega_\tau(d,2)$ for $\varepsilon>2$. Then the single-valued duality map $J:X\to X^\ast$ is norm-to-norm uniformly continuous on bounded subsets with modulus $\omega_\tau$, that is, for all $d,\varepsilon>0$ and $x,y\in X$ with $\norm{x},\norm{y}\leq d$ we have
\begin{equation*}
\norm{x-y}\leq\omega_\tau(d,\varepsilon)\to \norm{Jx-Jy}\leq \varepsilon.
\end{equation*}
\end{lemma}
\begin{theorem}[Theorem 2.1 of \cite{Lin(2004.0)} (cf. Remark 2.2)]
Let $X$ be uniformly smooth, and $A_1,A_2:D\to 2^D$ be two uniformly $\phi$-accretive operators for $D$ a nonempty, closed and convex subset of $X$, such that the ranges of $(I-A_1)$ and $(I-A_2)$ are bounded. Let $(\alpha_n),(\beta_n)$ be sequences in $[0,1)$ such that $\alpha_n,\beta_n\to 0$ and $\sum_{i=0}^\infty\alpha_i=\infty$. For any $f,x_0\in D$ let $(x_n)$ be generated via the Ishikawa-type scheme
\begin{equation*}
\begin{aligned}
x_{n+1}&=(1-\alpha_n)x_n+\alpha_n(f+u_n), \ \ \ u_n\in (I-A_1)y_n\\
y_n&=(1-\beta_n)x_n+\beta_n(f+v_n), \ \ \ v_n\in (I-A_2)x_n.
\end{aligned}
\end{equation*}
Then whenever the system of operator equations $\begin{cases}f\in A_1q\\ f\in A_2q\end{cases}$ has some solution $q\in D$, then $(x_n)$ converges strongly to $q$.
\end{theorem}
We now present a quantitative analysis of the above result, where for simplicity we set $f=0$.
\begin{theorem}
\label{thm-lin-quant}
Let $X$ be uniformly smooth with modulus $\tau$, and $A_1,A_2:D\to 2^X$ with $0\in A_iq$ for $i=1,2$ be uniformly accretive at zero, with $\Theta$ a modulus of uniform accretivity for $A_1$. Let $(\alpha_n),(\beta_n)$ be sequences in $[0,1)$ such that $\alpha_n,\beta_n\to 0$ with joint rate of convergence $\phi$ and $\sum_{i=0}^\infty\alpha_i=\infty$ with rate of divergence $r$. Suppose that $(x_n)$, $(y_n)$, $(u_n)$ and $(v_n)$ are sequences satisfying $x_n,y_n\in D(A)$ and
\begin{equation*}
\begin{aligned}
x_{n+1}&=(1-\alpha_n)x_n+\alpha_nu_n, \ \ \ u_n\in (I-A_1)y_n\\
y_n&=(1-\beta_n)x_n+\beta_nv_n, \ \ \ v_n\in (I-A_2)x_n.
\end{aligned}
\end{equation*}
Finally, suppose that $K_0,K_1\in (0,\infty)$ satisfy $\norm{w}<K_0$ for all $w\in R(I-A_i)$ for $i=1,2$ and $\norm{x_0-q}<K_1$. Then $\norm{x_n-q}\to 0$ as $n\to\infty$ with rate of convergence
\begin{equation*}
\begin{aligned}
&\Phi_{\Theta,\tau,\phi,r,K_0,K_1}(\varepsilon)\\
&:=r\left( \phi\left(\frac{1}{6K}\min\left\{\frac{\varepsilon}{2},\frac{3\Theta_K(\varepsilon/2)}{32K},\omega_\tau\left(K,\frac{\Theta_K(\varepsilon/{2})}{16K}\right)\right\}\right),\frac{K^2}{\Theta_K(\varepsilon/2)} \right)+1
\end{aligned}
\end{equation*}
for $K:=2K_0+K_1$ and $\omega_\tau$ as defined in Lemma \ref{lem-unifsmooth}.
\end{theorem}

\begin{proof}
To begin with, we claim that $\norm{x_n-q},\norm{y_n-q},
\norm{u_n-q},\norm{v_n-q}
<K:=2K_0+K_1$ for all $n\in\NN$ and moreover $\norm{y_n-x_{n+1}}\leq 3(\alpha_n+\beta_n)K<6K$, and therefore $\norm{y_n-x_{n+1}}\leq\delta$ for any 
$n\geq\phi(\tilde{\delta})$ with $\tilde{\delta}\leq \delta/6K.$ 
All of this is established entirely analogously to the beginning of the proof of Theorem \ref{thm-mn-quant}, which uses just the Ishikawa equations together with basic properties of normed spaces: Note our generalisation of the Ishikawa-type scheme to two maps is dealt with by the assumption that $K_0$ is a joint bound for the ranges of $(I-A_1)$ and $(I-A_2)$. 

Let us now define $j_n:=J(x_n-q)$ and $j'_n:=J(y_n-q)$ for each $n\in\NN$. By an application of Lemma \ref{lem-subdif} we have
\begin{equation}
\label{eqn-lin0}
\begin{aligned}
\norm{x_{n+1}-q}^2&\leq (1-\alpha_n)^2\norm{x_n-q}^2+2\alpha_n\pair{u_n-q,j_{n+1}}\\
&\leq (1-\alpha_n)^2\norm{x_n-q}^2+2\alpha_n\pair{u_n-q,j'_n}+2\alpha_n\pair{u_n-q,j_{n+1}-j'_n}\\
&\leq (1-\alpha_n)^2\norm{x_n-q}^2+2\alpha_n\pair{u_n-q,j'_n}+2\alpha_n Kc_n
\end{aligned}
\end{equation}
for $c_n:=\norm{j_{n+1}-j'_n}$, where for the last step we use $\pair{u_n-q,j_{n+1}-j'_n}\leq \norm{u_n-q}\cdot \norm{j_{n+1}-j'_n}\leq Kc_n$. An analogous calculation yields
\begin{equation}
\label{eqn-lin1}
\begin{aligned}
\norm{y_n-q}^2&\leq (1-\beta_n)^2\norm{x_n-q}^2+2\beta_n\pair{v_n-q,j'_n}\\
&\leq (1-\beta_n)^2\norm{x_n-q}^2+2\beta_n\pair{v_n-q,j_n}+2\beta_nKd_n
\end{aligned}
\end{equation}
for $d_n:=\norm{j'_n-j_n}$. Now, by accretivity of $A_2$ at zero we have $\pair{v_n-q,j_n}\leq\norm{x_n-q}^2$ and substituting this into (\ref{eqn-lin1}) we get
\begin{equation}
\label{eqn-lin2}
\norm{y_n-q}^2\leq (1+\beta^2_n)\norm{x_n-q}^2+2\beta_nKd_n.
\end{equation}
For the remainder of the proof we fix some $\varepsilon>0$ and suppose that $\varepsilon<\norm{x_{n+1}-q}^2$. We now suppose that
\[ (*)\ n\geq N(\varepsilon):=
\phi\left(\min\left\{\tfrac{\sqrt{\varepsilon}}{12K}, 
\tfrac{\Theta_K(\sqrt{\varepsilon}/2)}{64K^2},
\omega_{\tau}\left(K,\tfrac{\Theta_K(\sqrt{\varepsilon}/2)}{16K}\right)/6K\right\}\right).\] 
Then, in particular, 
$\norm{y_n-x_{n+1}}\leq\sqrt{\varepsilon}/2$. This then implies that
\begin{equation*}
\sqrt{\varepsilon}<\norm{x_{n+1}-q}\leq \norm{x_{n+1}-y_n}+\norm{y_n-q}\leq \sqrt{\varepsilon}/2+\norm{y_n-q}
\end{equation*}
and so $\norm{y_n-q}\in [\sqrt{\varepsilon}/2,K]$. By Lemma \ref{lem-pseudo} we then have $\pair{u_n-q,j'_n}\leq\norm{y_n-q}^2-\Theta_K(\sqrt{\varepsilon}/2)$, and thus using (\ref{eqn-lin2}):
\begin{equation*}
\label{eqn-lin3}
\pair{u_n-q,j'_n}\leq (1+\beta^2_n)\norm{x_n-q}^2+2\beta_nKd_n-\Theta_K(\sqrt{\varepsilon}/2).
\end{equation*}
Finally, substituting this into (\ref{eqn-lin0}) we obtain
\begin{equation}
\label{eqn-lin4}
\begin{aligned}
&\norm{x_{n+1}-q}^2\\
&\leq (1+\alpha_n^2)\norm{x_n-q}^2+2\alpha_n\beta^2_n\norm{x_n-q}^2+4\alpha_n\beta_nKd_n-2\alpha_n\Theta_K(\sqrt{\varepsilon}/2)+2\alpha_n Kc_n\\
&<\norm{x_n-q}^2+\alpha_n^2K^2+2\alpha_n\beta^2_nK^2+4\alpha_n\beta_nKd_n-2\alpha_n\Theta_K(\sqrt{\varepsilon}/2)+2\alpha_n Kc_n\\
&=\norm{x_n-q}^2-\alpha_n\cdot \delta
\end{aligned}
\end{equation}
for
\begin{equation*}
\delta:=2\Theta_K(\sqrt{\varepsilon}/2)-(\alpha_nK^2+2\beta^2_nK^2+4\beta_nKd_n+2 Kc_n).
\end{equation*}
Define $\delta_0:=\Theta_K(\sqrt{\varepsilon}/2)/8.$ 
Then $(*)$ implies $\alpha_n,\beta_n\leq \delta_0/8K^2$ and so in turn 
$\alpha_nK^2<\delta_0$, $2\beta_n^2 K^2<\delta_0$ (using $\beta_n<1$), and $4\beta_nKd_n< \delta_0$. For the latter note that 
\begin{equation*}
d_n=\norm{j'_n-j_n}\leq\norm{j_n}+\norm{j'_n}=\norm{x_n-q}+\norm{y_n-q}<2K.
\end{equation*}
Finally, let $\omega_\tau$ be defined as in Lemma \ref{lem-unifsmooth}. Then 
again by $(*)$ 
we have $\norm{(x_{n+1}-q)-(y_n-q)}=\norm{x_{n+1}-y_n}\leq \omega_\tau(K,\delta_0/2K)$ and thus by Lemma \ref{lem-unifsmooth} it follows that $c_n=\norm{j_{n+1}-j'_n}\leq \delta_0/2K$ and thus $2Kc_n\leq\delta_0$. 

Putting all this together we conclude that if $\varepsilon <\norm{x_{n+1}-q}^2$ and $n\ge N(\varepsilon)$ we have $\delta\geq \Theta_K(\sqrt{\varepsilon}/2)$ and thus by (\ref{eqn-lin4}):
\begin{equation*}
\norm{x_{n+1}-q}^2\leq\norm{x_n-q}^2-\alpha_n\cdot \varphi(\varepsilon)
\end{equation*}
for $\varphi(\varepsilon):=\Theta_K(\sqrt{\varepsilon}/2)$. 
\\ 
We can now apply Lemma \ref{lem-tech} as in the proof of Theorem \ref{thm-mn-quant} for $\theta_n:=\norm{x_n-q}^2$ on parameters $K^2,r,\varphi$ and $N$ to obtain the stated rate of convergence.
\end{proof}

\begin{remark}
The precise statement of Theorem \ref{thm-lin-quant} is consistent with Remark 2.2 of \cite{Lin(2004.0)}, in that we only require a \emph{modulus} of uniform accretivity for one of the operators (though we require both to be accretive).
\end{remark}

\section{Acknowledgement}
This work has been supported by the German Science Foundation DFG 
(Project KO 1737/6-1).

\end{document}